\pdfoutput=1
\documentclass[12pt]{article}

\usepackage[OT2, T1]{fontenc}
\usepackage[russian, ngerman, USenglish]{babel}

\usepackage{amssymb,amsmath,amsthm,mathrsfs,amsfonts,latexsym,xcolor,graphics,graphicx}

\usepackage[bookmarks=false,colorlinks=true,linkcolor=blue,urlcolor=blue,citecolor=blue,anchorcolor=blue]{hyperref}
\newcommand{\doi}[1]{\href{http://dx.doi.org/#1}{doi:#1}}

\definecolor{darkgreen}{rgb}{0,.5,0.2}

\usepackage{tikz}

\DeclareMathOperator{\dist}{dist}

\newcommand{\eqdef}{\mathbin{\raisebox{0.4pt}{\ensuremath{:}}\hspace*{-1mm}=}}

\newcommand{\bC}{\mathbb C}
\newcommand{\bT}{\mathbb T}
\newcommand{\bZ}{\mathbb Z}
\newcommand{\bN}{\mathbb N}
\newcommand{\bR}{\mathbb R}

\newcommand{\eps}{\varepsilon}

\newcommand{\al}{\alpha}
\newcommand{\be}{\beta}

\newcommand{\de}{\delta}

\newcommand{\nf}{\infty}

\newcommand{\la}{\lambda}

\newcommand{\ph}{\varphi}

\newcommand{\si}{\sigma}
\newcommand{\tht}{\theta}

\newcommand{\Range}{\mathcal{R}}

\newcommand{\supp}{\operatorname{supp}}
\newcommand{\cF}{\mathcal{F}}
\newcommand{\cN}{\mathcal{N}}

\newcommand{\AbstractProb}{\mu} % to speak about general properties
\newcommand{\LimProb}{\Lambda} % the limit probability measure
\newcommand{\prob}{\mu} % we suppose that $\prob_n$ tends to $\LimProb$ in distribution
\newcommand{\LimFunction}{X} % after that we substitute $\LimProb$ by $(\LimFunction,\ProbInDomain)$
\newcommand{\Domain}{\Omega}
\newcommand{\ProbInDomain}{\operatorname{P}}

\newcommand{\tuplesize}{d}

\newtheorem{theorem}{Theorem}[section]
\newtheorem{corollary}[theorem]{Corollary}

\newtheorem{proposition}[theorem]{Proposition}

\newtheorem{example}[theorem]{Example}

\newcommand{\vsk}{\vspace{1mm}}
\newcommand{\vsg}{\vspace{2mm}}

\begin{document}

{\Large \bf From convergence in distribution to uniform \\ convergence}

\vspace{4mm}
{\large \bf J. M. Bogoya, A. B\"ottcher, and E. A. Maximenko}

\vspace{4mm}
{\em Dedicated with thanks to Sergei Grudsky, who has left his imprint on all the three of us, on his sixtieth birthday}

{\footnotesize

\vspace{8mm}
{\bf Abstract}
We present conditions that allow us to pass
from the convergence of probability measures in distribution
to the uniform convergence of the associated quantile functions. Under these conditions,
one can in particular pass from the asymptotic distribution of collections of real numbers,
such as the eigenvalues of a family of $n$-by-$n$ matrices as $n$ goes to infinity,
to their uniform approximation by the values of the quantile function at equidistant points.
For Hermitian Toeplitz-like matrices, convergence in distribution is ensured by
theorems of the Szeg\H{o} type. Our results transfer these convergence theorems into uniform
convergence statements.

\vspace{2mm}

\vspace{2mm}
}

\section{Introduction and main results}\label{sec:intro}

%\vsg
It was exactly 100 years ago when Szeg\H{o} published his seminal paper \cite{Szeg1} on Toeplitz determinants.
Only five years later, his theorem on the asymptotic distribution of the eigenvalues of Hermitian Toeplitz matrices
appeared~\cite{Szeg2}. Since then spectral properties of Toeplitz matrices, in particular the collective behavior
of eigenvalues, have been extensively studied
by many authors; see, for example, the books~\cite{BG2005,BS1999}. However, it was only recently that
asymptotic formulas for individual eigenvalues inside the spectrum backed in the interest;
see~\cite{BBGM2015simpleloop,BGM2010,DIK2012}. This topic is still in its infancy, because the results
so far available cover very particular classes of generating functions only.

\vsk
In our paper \cite{BBGM2015maximum} with Sergei Grudsky, which was in fact inspired by the
papers~\cite{BBGM2015simpleloop,Trench2012},
we proved a result on the uniform approximation of the singular values of Toeplitz matrices,
which are the eigenvalues in the case of positive definite Hermitian matrices.
The purpose of the present paper is to simplify some proofs from \cite{BBGM2015maximum} and
to put the approach into a more abstract setting, thus
extending the range of possible applications.

{\footnotesize

\vspace{3mm}
----------------------------------------------------------------------------------------------------------------------------------

The third author's research was partially supported by project
IPN-SIP 20150422 (Instituto Polit\'{e}cnico Nacional, Mexico).

-----------------------------------------------------------------------------------------------------------------------------------

\vsg
J. M. Bogoya,
Pontificia Universidad Javeriana, Departamento de Matem\'aticas, 01110 Bogot\'a D.C., Colombia

e-mail: jbogoya@javeriana.edu.co

\vsk
A. B\"ottcher,
Fakult\"at f\"ur Mathematik,
Technische Universit\"at Chemnitz,
09107 Chemnitz,
Germany

e-mail: aboettch@mathematik.tu-chemnitz.de

\vsk
E. A. Maximenko,
Instituto Polit\'ecnico Nacional, Escuela Superior de F\'{\i}sica y Matem\'aticas, 07730 Ciudad de M\'exico, Mexico

e-mail: maximenko@esfm.ipn.mx
}

\newpage
A probability measure $\AbstractProb$ is called a \emph{Borel probability measure} on $\bR$
if its domain contains the Borel $\sigma$-algebra over $\bR$.
Given a Borel probability measure $\AbstractProb$ on $\bR$,
the corresponding \emph{cumulative distribution function} $F_\AbstractProb\colon\bR\to[0,1]$
and \emph{quantile function} $Q_\AbstractProb\colon(0,1)\to\bR$ are defined by
\begin{equation}\label{eq:defFQ}
F_\AbstractProb(v)\eqdef\AbstractProb(-\infty,v],\qquad
Q_\AbstractProb(p)\eqdef\inf\{v\in\bR\colon\ F_\AbstractProb(v)\ge p\}.
\end{equation}

\vsk
The \emph{support} of $\AbstractProb$ is the set
\begin{equation}\label{eq:defsupport}
\supp(\AbstractProb)\eqdef\{v\in\bR\colon\AbstractProb(v-\eps,v+\eps)>0\quad\forall\eps>0\}.
\end{equation}
If $\AbstractProb$ has a bounded support,
then the function $Q_\AbstractProb$ has finite limits at the points~$0$ and~$1$,
and we extend $Q_\AbstractProb$ to $[0,1]$ by continuity.

\vsk
Herewith our first main result.

\begin{theorem}\label{thm:measures}
Let $\LimProb$ be a Borel probability measure on $\bR$
and $(\prob_n)_{n=1}^\infty$ be a sequence of Borel probability measures on $\bR$
that converges to $\LimProb$ in distribution, i.e.,
\begin{equation}\label{eq:probconvdistr}
\lim_{n\to\nf} \int_\bR \ph\,d\prob_n = \int_\bR \ph\,d\LimProb
\end{equation}
for every $\ph\in C_b(\bR)$.
Moreover, suppose that $\supp(\LimProb)$ is a bounded and connected set
and that $\supp(\prob_n)\subseteq\supp(\LimProb)$ for every $n\in\bN$.
Then the sequence $(Q_{\prob_n})_{n=1}^\infty$ converges uniformly to $Q_\LimProb$:
\begin{equation}\label{eq:probconvuni}
\lim_{n\to\nf}\sup_{p\in[0,1]}|Q_{\prob_n}(p)-Q_\LimProb(p)|=0.
\end{equation}
\end{theorem}

In this theorem, the class $C_b(\bR)$ of  bounded continuous functions of $\bR$ to $\bC$
can be substituted by the class $C_c(\bR)$ of continuous functions with compact support,
because $\supp(\LimProb)$ is supposed to be a segment of $\bR$.

\vsk
Theorem \ref{thm:measures} makes precise what we mean by passing from convergence in
distribution to uniform convergence. We emphasize that this passage is based on two
assumptions: first, $\supp(\LimProb)$ is required to be a segment and secondly,
all supports $\supp(\prob_n)$ must be contained in this segment. These assumptions are
not caused by our proof but are essential. In~\cite{BBGM2015maximum} we considered a concrete
realization of the setting and showed that the conclusion of Theorem~\ref{thm:measures}
is no longer true if one of the two assumptions is violated.

\vsk
We now specialize the measures $\prob_n$
to be discrete measures associated to collections of real numbers.
On the other hand, we allow $d\LimProb$ to be of the form $X d\ProbInDomain$
with a measurable function $\LimFunction$ on an abstract probability space $(\Omega,\cF,\ProbInDomain)$.
In this setting, one makes the following definition
(see \cite{GSV2015}, for example).
Let $(\tuplesize(n))_{n=1}^\infty$ be a sequence of positive integer numbers tending to infinity
and let
\[\al=\left(\al^{(n)}_1,\ldots,\al^{(n)}_{\tuplesize(n)}\right)_{n=1}^\infty\]
be a sequence of collections of real numbers.
In addition, let $(\Domain,\cF,\ProbInDomain)$ be a probability space
and $\LimFunction\colon\Domain\to\bR$ be an $\cF$-measurable function.
The sequence $\al$ is said to be \emph{asymptotically distributed} as $(\LimFunction,\ProbInDomain)$
if, for every function $\ph\in C_c(\bR)$,
\begin{equation}\label{eq:def_asympt_distr}
\lim_{n\to\nf} \frac{1}{\tuplesize(n)}\sum_{j=1}^{\tuplesize(n)} \ph(\alpha^{(n)}_j) = \int_\Omega \ph\circ\LimFunction\,d\ProbInDomain.
\end{equation}

\vsk
Given a probability space $(\Domain,\cF,\ProbInDomain)$
and an $\cF$-measurable function $\LimFunction\colon\Domain\to\bR$,
we denote by $\Range(\LimFunction)$, $F_{\LimFunction}$, and $Q_{\LimFunction}$
the essential range of $\LimFunction$, the cumulative distribution function,
and the quantile function associated to $\LimFunction$:
\begin{eqnarray*}
& & \Range(\LimFunction)\eqdef\{v\in\bR\colon\ProbInDomain(\LimFunction^{-1}(v-\eps,v+\eps))>0\quad\forall\eps>0\},\\
& & F_\LimFunction(v)\eqdef\ProbInDomain(\LimFunction^{-1}(-\infty,v]),\quad
Q_\LimFunction(p)\eqdef\inf\{v\in\bR\colon\ F_\LimFunction(v)\ge p\}.
\end{eqnarray*}
In this situation we have our second main result.

\begin{theorem}\label{thm:tuples}
Let a sequence $\al$ of collections of real numbers be
asymptotically distributed as $(\LimFunction,\ProbInDomain)$.
Suppose $\Range(\LimFunction)$ is connected and bounded
and suppose also that, for each $n\in\bN$, the numbers $\al^{(n)}_1,\ldots,\al^{(n)}_{\tuplesize(n)}$ belong to $\Range(\LimFunction)$
and are ordered in the ascending manner:
\begin{equation}\label{ascending}
\al^{(n)}_1\le \al^{(n)}_2 \le \dots \le \al^{(n)}_{\tuplesize(n)}.
\end{equation}
Then
\begin{equation}\label{eq:convtuplessup}
\lim_{n\to\nf} \max_{1\le j\le \tuplesize(n)}\sup_{\frac{j-1}{\tuplesize(n)}\le u\le\frac{j}{\tuplesize(n)}} |\alpha^{(n)}_j-Q_\LimProb(u)| = 0.
\end{equation}
In particular,
\begin{equation}\label{eq:convtuplesnodes}
\lim_{n\to\nf} \max_{1\le j\le \tuplesize(n)} |\al^{(n)}_j-Q_\LimProb(j/\tuplesize(n))| = 0.
\end{equation}
\end{theorem}

We finally consider the special case where $\Domain$ is a finite interval in $\bR$,
$\ProbInDomain$ is the normalized Lebesgue measure on $\Domain$,
and $\LimFunction$ is Riemann integrable. In that case we prove the following theorem,
which reveals that the values of $Q_{\LimProb}$ at equidistant points
can be replaced by the ordered values of the original function $\LimFunction$ at some points of $\Domain$.

\begin{theorem}\label{thm:Riemann}
Let $\Domain$ be a bounded interval of $\bR$, $\ProbInDomain$ be the normalized Lebesgue measure on $\Domain$,
$\LimFunction\colon\Domain\to\bR$ be a Riemann integrable function with connected essential range,
$\al=(\al^{(n)}_1,\ldots,\al^{(n)}_{\tuplesize(n)})_{n=1}^{\nf}$ be a sequence of collections of real numbers
asymptotically distributed as $(\LimFunction,\ProbInDomain)$ such that, for every $n\in\bN$,
the numbers $\al^{(n)}_1,\ldots,\al^{(n)}_{\tuplesize(n)}$
satisfy \eqref{ascending} and belong to $\Range(\LimFunction)$.
Furthermore, for every $n\in\bN$, let $\xi^{(n)}_1,\ldots,\xi^{(n)}_{\tuplesize(n)}$
be any points belonging to the different parts of the canonical $\tuplesize(n)$-partition of the interval $\Domain$,
let $v^{(n)}_1,\ldots,v^{(n)}_{\tuplesize(n)}$ be the values of $\LimFunction$ at these points,
and let $\si_n$ be a permutation of $\{1,\ldots,\tuplesize(n)\}$ such that
\[
v^{(n)}_{\si_n(1)}\le\dots\le v^{(n)}_{\si_n(\tuplesize(n))}.
\]
Then
\[
\lim_{n\to\nf}\max_{1\le j\le \tuplesize(n)} |\al^{(n)}_j-v^{(n)}_{\si_n(j)}| = 0.
\]
\end{theorem}

Here is an outline of the paper.
After recalling some general continuity properties of the quantile function
in Section~\ref{sec:qu}, we prove the main results stated above in~Section~\ref{sec:proofs}.
In Section~\ref{sec:applToeplitz} we embark on some applications of the theorems
to the singular values and eigenvalues of Toeplitz-like matrices,
and in Section~\ref{sec:applsequences} we give examples of applications
to problems from beyond the matrix world.

\section{Continuity of the quantile function}
\label{sec:qu}

In this section we record some continuity properties of the quantile function.
This section is very close to the Section~2 in~\cite{BBGM2015maximum}, but we changed some technical details.

\vsk
Throughout this section we suppose that $\AbstractProb$ is a Borel probability measure on $\bR$ with bounded support.
We use the simplified notation $F$ and $Q$ for the functions $F_\AbstractProb$ and $Q_\AbstractProb$, correspondingly.
Recall that these functions are defined by \eqref{eq:defFQ}.

\vsk
It is well known and readily verified
that $F$ and $Q$ are monotonically increasing (in the non-strict sense),
that $F$ is continuous from the right, that the infimum in the definition of $Q(p)$
belongs to the set $\bigl\{v\in\bR\colon p\le F(v)\bigr\}$
and therefore is the minimum of this set,
and that $Q$ is continuous from the left.
The one-sided limits $F(v^-)$ and $Q(p^+)$
exist for each $v\in\bR$ and each $p\in(0,1)$.
It follows from the definition of $Q$ that, for every $v$ in $\bR$,
\begin{equation}\label{eq:QFv_le_v}
Q(F(v))\le v
\end{equation}
and
\begin{equation}\label{eq:QFvplus_ge_v}
Q(F(v)^+)\ge v.
\end{equation}
We thoroughly work under the assumtion that  $\supp(\AbstractProb)$ is compact.
We denote by $\al$ and $\be$ the minimum and maximum of $\supp(\AbstractProb)$:
\begin{equation}\label{eq:minmaxsupp}
\al \eqdef \min\supp(\AbstractProb),\qquad \be \eqdef \max\supp(\AbstractProb).
\end{equation}
We write $F(-\infty)$ and $F(+\infty)$ for
the limits of $F(v)$ as $v\to-\infty$ and $v\to+\infty$, respectively.
The next proposition deals with $F$ near $\al$ and $\be$ and with $Q$ near $0$ and~$1$.

\begin{proposition}\label{prop:FQextreme}
The functions $F$ and $Q$ have the following properties.

{\rm (a)} $F(v)=0$ for every $v$ in $[-\nf,\al)$.

{\rm (b)} $F(\al)=\AbstractProb\{\al\}$.

{\rm (c)} $0<F(v)<1$ for every $v$ in $(\al,\be)$.

{\rm (d)} $F(v)=1$ for every $v$ in $[\be,+\nf]$.

{\rm (e)} $\al\le Q(p)\le\be$ for every $p$ in $(0,1)$.

{\rm (f)} $Q(0^+)=\al$, $Q(1^-)=\be$.
\end{proposition}

\begin{proof}
Properties (a) to (d) follow directly from the definition of $F$, $\al$, and $\be$.
Given $p\in(0,1)$, the inequality $Q(p)\ge\al$ results from (a),
while the inequality $Q(p)\le\be$ is a consequence of (c). This proves (e),
and we are left with (f).

\vsk
We first turn to the limit of $Q(p)$ as $p\to0^+$.
Given $\eps>0$, put $q=F(\al+\eps)$.
Then $q>0$, and for every $p$ in $(0,q]$ we can apply \eqref{eq:QFv_le_v} to obtain
\[
\al \le Q(p) \le Q(q) = Q(F(\al+\eps)) \le \al+\eps.
\]
This implies that $Q(0^+)=\al$.
We now consider the limit of $Q(p)$ as $p\to1^-$.
For $\eps>0$, we put $q=F(\be-\eps)$.
Then $q<1$, and for every $p\in(q,1)$ we infer from \eqref{eq:QFvplus_ge_v} that
\[
\be-\eps \le Q(F(\be-\eps)^+) \le Q(p) \le \be.
\]
Thus $Q(1^-)=\be$.
\end{proof}

We extend $Q$ by continuity to $[0,1]$:
$
Q(0):=Q(0^+)=\al$ and $Q(1):=Q(1^-)=\be$.
Note that we do not define $Q(0)$ by putting $p=0$ into \eqref{eq:defFQ},
because the corresponding value would be $-\nf$.

\vsk
Here are some well known or easily verifiable relations between $F$ and $Q$.

\begin{proposition}\label{prop:FQ}
The following are true.

{\rm (a)} $Q(F(v))\le v$ for every $v\in\bR$.

{\rm (b)} $F(Q(u))\ge u$ for every $u\in[0,1]$.

{\rm (c)} Let $u\in[0,1]$ and $v\in\bR$. Then
$Q(u)\le v$ if and only if $u\le F(v)$.

{\rm (d)} If $v_1,v_2\in\bR$
and $F(v_1)<F(v_2)$, then
$v_1<Q(F(v_2))\le v_2$.
\end{proposition}

\begin{proposition}\label{distrQ}
The distribution function of $Q$ is $F$, i.e., for every $v\in\bR$,
\[
\mu_{\bR}\{u\in[0,1]\colon Q(u)\le v\}=F(v),
\]
where $\mu_{\bR}$ stands for the Lebesgue measure on $\bR$.
\end{proposition}

The next criterion implies in particular
that the connectedness of $\supp(\AbstractProb)$
is equivalent to the continuity of $Q$.
This condition plays a crucial role in this paper.
For a proof, see \cite{BBGM2015maximum}.

\begin{proposition}\label{critcont}
The following conditions are equivalent:

{\rm(i)} $\supp(\AbstractProb)$ is connected, i.e., $\supp(\AbstractProb)=[\al,\be]$.

{\rm(ii)} $F$ is strictly increasing on $[\al,\be]$.

{\rm(iii)} $Q(F(v))=v$ for every $v\in[\al,\be]$.

{\rm(iv)} $Q([0,1])=[\al,\be]$.

{\rm(v)} $Q$ is continuous on $[0,1]$.

\end{proposition}

\begin{corollary}\label{Qunicont}
Let $\AbstractProb$ be a Borel probability measure on $\bR$
such that $\supp(\AbstractProb)$ is bounded and connected.
Then $F$ is strictly increasing and $Q$ is uniformly continuous on $[0,1]$.
\end{corollary}

Here is a result concerning Riemann integrable functions. It is one of the basic ingredients to the proof of Theorem~\ref{thm:Riemann}.
It was proved in \cite{BBGM2015maximum} in slightly different notation.

\begin{proposition}\label{prop:RM}
Let $\Domain \subset \bR$ be a bounded interval, $\ProbInDomain$ be the normalized Lebesgue measure on $\Domain$,
and $\LimFunction\colon\Domain\to\bR$ be a Riemann integrable function with connected essential range.
For every $n\in\bN$, let $\xi^{(n)}_1,\ldots,\xi^{(n)}_{\tuplesize(n)}$,
$v^{(n)}_1,\ldots,v^{(n)}_{\tuplesize(n)}$, and $\si$, be as in Theorem~\ref{thm:Riemann}.
Then
\begin{equation}
\lim_{n\to\nf}\max_{1\le j\le \tuplesize(n)}\left|v^{(n)}_{\si_n(j)} - Q_\LimFunction(j/\tuplesize(n))\right|=0. \label{eq:RM}
\end{equation}
\end{proposition}

\section{Proofs of the main results}
\label{sec:proofs}

The following proposition is a special version of Alexandroff's criterion for convergence in distribution,
which is also known as the portmanteau lemma
(see, for example, Sections~8.1 and 8.2 of~\cite{Bogachev2007} or
Lemma~2.2 and Lemma~21.2 of~\cite{Vaart}).
Proofs of the equivalence (i)$\Leftrightarrow$(ii) can be found in
Sections~8.1 and 8.2 of~\cite{Bogachev2007} or Lemma~2.2 of~\cite{Vaart}.
We remark that this equivalence was established by A.~D.~Alexandroff
in the more abstract context of metric spaces.
The equivalence (ii)$\Leftrightarrow$(iii) is elementary:
see Lemma~21.2 of \cite{Vaart}.

\begin{proposition}\label{prop:critconvdistr}
Let $\LimProb$ be a Borel probability measure on $\bR$
and $(\prob_n)_{n=1}^\nf$ be a sequence of Borel probability measures on $\bR$.
Then the following conditions are equivalent.

{\rm(i)} For every $\ph\in C_b(\bR)$, \eqref{eq:probconvdistr} holds.

{\rm(ii)} $\displaystyle\lim_{n\to\nf} F_{\prob_n}(v) = F_{\LimProb}(v)$
for every point $v\in\bR$ at which $F_{\LimProb}$ is continuous.

{\rm(iii)} $\displaystyle\lim_{n\to\nf} Q_{\prob_n}(p) = Q_{\LimProb}(p)$
for every point $p\in(0,1)$ at which $Q_{\LimProb}$ is continuous.
\end{proposition}

The next result says that pointwise convergence on a segment,
jointly with monotonicity and continuity, imply uniform convergence.

\begin{proposition}\label{prop:monotone_implies_uniform}
Let $g\colon[0,1]\to\bR$ be a continuous function
and $(f_n)_{n=1}^\nf$ be a sequence of functions $[0,1]\to\bR$ such that
the function $f_n$ is increasing {\rm (}in the non-strict sense{\rm )} for every $n \in \bN$
and
$f_n(p)\to g(p)$ as $n \to \infty$ for every $p\in[0,1]$.
Then
\[
\lim_{n\to\nf}\sup_{p\in[0,1]}|f_n(p)-g(p)|=0.
\]
\end{proposition}

\begin{proof}
Let $\eps>0$.
Since $g$ is uniformly continuous on $[0,1]$,
we can select a positive number $\de$ such that
\begin{equation}\label{guni}
|g(p_1)-g(p_2)|\le\frac{\eps}{2}\;\:\mbox{whenever}\;\:p_1,p_2\in[0,1]\;\:\mbox{and}\;\:|p_1-p_2|\le\de.
\end{equation}
Choose $m\in\{1,2,\ldots\}$ such that $1/m\le\de$.
Using the convergence $f_n(p)\to g(p)$ at the points $p=j/m$, $j=0,1,\ldots,m$, we
find an $n_0\in\bN$ such that, for every $n\ge n_0$ and every $j\in\{0,\ldots,m\}$,
\begin{equation}\label{fapproxg}
\left|f_n\left(\frac{j}{m}\right)-g\left(\frac{j}{m}\right)\right|<\frac{\eps}{2}.
\end{equation}
Now let $n\ge n_0$ and $p\in[0,1]$.
Pick $j\in\{0,\ldots,m-1\}$ such that $j/m\le p\le (j+1)/m$.
Then the monotonicity of $f_n$ together with the inequalities \eqref{fapproxg} and \eqref{guni} implies that
\[
f_n(p)\le f_n\left(\frac{j+1}{m}\right) < g\left(\frac{j+1}{m}\right) + \frac{\eps}{2} < g(p)+\eps.
\]
In a similar manner,
\[
f_n(p)\ge f_n\left(\frac{j}{m}\right) > g\left(\frac{j}{m}\right)-\frac{\eps}{2} > g(p)-\eps,
\]
which completes the proof.
\end{proof}

\vsg
\begin{proof}[Proof of Theorem~\ref{thm:measures}.]
By Corollary~\ref{Qunicont}, $Q_{\LimProb}$ is uniformly continuous on $[0,1]$.
Therefore, by Proposition~\ref{prop:critconvdistr},
$Q_{\prob_n}$ pointwisely converges to $Q_{\LimProb}$ on $(0,1)$. We are so left with the points $0$ and $1$.

\vsk
We consider the situation at the point $0$. This is the place where
the assumption that $\supp(\prob_n)\subseteq \supp(\LimProb)$ makes its debut.
It implies that
\[
Q_{\prob_n}(0)=\inf\supp(\prob_n)\ge \inf\supp(\LimProb) = Q_{\LimProb}(0).
\]
Then, given $\eps>0$, there is a $\de>0$ such that $Q_{\LimProb}(\de)<Q_{\LimProb}(0)+\eps/2$
and an $n_0\in\bN$ such that $|Q_{\prob_n}(\de)-Q_{\LimProb}(\de)|<\eps/2$ for every $n\ge n_0$.
Consequently, for every $n\ge n_0$,
\[
Q_{\LimProb}(0)\le Q_{\prob_n}(0)\le Q_{\prob_n}(\de)<Q_{\LimProb}(\de)+\eps<Q_{\LimProb}(0)+\eps,
\]
whence $Q_{\prob_n}(0)\to Q_{\LimProb}(0)$.
The convergence at the point $1$ can be proved in a similar manner.
Thus, $Q_{\prob_n}$ pointwisely converges to $Q_{\LimProb}$ on $[0,1]$.
Proposition~\ref{prop:monotone_implies_uniform} now implies that the convergence is uniform.
\end{proof}

\vsg
\begin{proof}[Proof of Theorem~\ref{thm:tuples}]
We simply translate Theorem~\ref{thm:tuples}
into the language of Theorem~\ref{thm:measures}.

\vsk
First step.
For each $n\in\bN$, we denote by $\prob_n$ the normalized counting measure associated to the tuple
$(\al^{(n)}_1,\ldots,\al^{(n)}_{\tuplesize(n)})$, i.e., for every Borel subset $B$ of $\bR$, we put
\[
\prob_n(B) = \frac{\#\{j\in\{1,\ldots,\tuplesize(n)\}\colon\ \al^{(n)}_j\in B\}}{\tuplesize(n)}.
\]
In other words, $\prob_n$ is nothing but the arithmetic mean of the Dirac measures
concentrated at the points $\al^{(n)}_1,\ldots,\al^{(n)}_{\tuplesize(n)}$:
\[
\prob_n = \frac{1}{\tuplesize(n)}\sum_{j=1}^{\tuplesize(n)} \de_{\al^{(n)}_j}.
\]
Since the tuple $(\al^{(n)}_1,\ldots,\al^{(n)}_{\tuplesize(n)})$ is ordered,
\[
F_{\prob_n}(v) = \frac{\max \{j\in\{1,\ldots,\tuplesize(n)\}\colon\ \al^{(n)}_j \le v\}}{\tuplesize(n)}
\]
and
\begin{equation}\label{eq:Qal}
Q_{\prob_n}(j/\tuplesize(n)) = \al^{(n)}_j.
\end{equation}
Second step.
Denote by $\LimProb$ the pushforward measure on $\bR$ associated to $\ProbInDomain$ and $\LimFunction$,
i.e., for every Borel subset $B$ of $\bR$, put
\[
\LimProb(B) = \ProbInDomain(\LimFunction^{-1}(B)).
\]
Then $F_\LimFunction=F_\LimProb$, $Q_\LimFunction=Q_\LimProb$, and
\begin{align*}
\Range(\LimFunction)
&= \{v\in\bR\colon\ProbInDomain(\LimFunction^{-1}((v-\eps,v+\eps)))>0\quad\forall\eps>0\}
\\
&= \{v\in\bR\colon\LimProb(v-\eps,v+\eps)>0\quad\forall\eps>0\}
= \supp(\LimProb).
\end{align*}
Third step.
Since $\Range(\LimFunction)$ is bounded and the points $\al^{(n)}_j$ belong to $\Range(\LimFunction)$,
the limit relation \eqref{eq:def_asympt_distr} holds not only for every $\ph\in C_c(\bR)$,
but for every $\ph\in C_b(\bR)$,
i.e., $\prob_n$ converges to $\LimProb$ in distribution.
By Theorem~\ref{thm:measures}, $Q_{\prob_n}$ uniformly converges to $Q_{\LimFunction}$.
Using \eqref{eq:Qal} we conclude that
\begin{eqnarray*}
\max_{1\le j\le \tuplesize(n)} |\alpha^{(n)}_j-Q_\LimProb(j/\tuplesize(n))|
& = & \max_{1\le j\le \tuplesize(n)} |Q_{\prob_n}(j/\tuplesize(n))-Q_\LimProb(j/\tuplesize(n))|\\
& \le & \sup_{p\in[0,1]} |Q_{\prob_n}(p) - Q_\LimProb(p)|,
\end{eqnarray*}
which completes the proof of  \eqref{eq:convtuplesnodes}.
Finally, from the uniform continuity of $Q_{\LimProb}$ we obtain
\[
\lim_{n\to\nf} \max_{1\le j\le \tuplesize(n)}
\sup_{\frac{j-1}{\tuplesize(n)}\le u\le\frac{j}{\tuplesize(n)}} |Q_\LimProb(j/\tuplesize(n))-Q_\LimProb(u)| = 0,
\]
which jointly with \eqref{eq:convtuplesnodes} yields \eqref{eq:convtuplessup}.
\end{proof}

\vsg
\begin{proof}[Proof of Theorem~\ref{thm:Riemann}]
The assertion of this theorem is immediate from Theorem~\ref{thm:tuples} and Proposition~\ref{prop:RM}.
\end{proof}

\section{Applications to Toeplitz-like matrices}
\label{sec:applToeplitz}

Given a matrix $A\in\mathbb{C}^{\tuplesize\times \tuplesize}$,
we denote by $s_{1}(A),\ldots,s_{\tuplesize}(A)$ its singular values
written in the ascending order,
$
s_{1}(A)\le \dots \le s_{\tuplesize}(A),
$
and for a Hermitian matrix $A\in\mathbb{C}^{\tuplesize\times \tuplesize}$,
we let $\la_{1}(A),\ldots,\la_{\tuplesize}(A)$ stand for its eigenvalues
written in the ascending order, taking multiplicities into account,
$
\la_{1}(A)\le \dots \le \la_{\tuplesize}(A).
$
In accordance with the definition of asymptotic distribution given in Section~1,
we adopt the following terminology.
Let $(A_n)_{n=1}^\nf$ be a sequence of square complex matrices,
denote the order of $A_n$ by $\tuplesize(n)$, and suppose that $\tuplesize(n)\to\infty$ as $n\to\nf$.
Let $(\Domain,\cF,\ProbInDomain)$ be a probability space
and let $\LimFunction\colon \Domain\to\bR$ be an $\cF$-measurable function.
If the sequence of tuples $(s_{1}(A_n),\dots,s_{\tuplesize(n)}(A_n))_{n=1}^{\nf}$
is asymptotically distributed as $(\LimFunction,\ProbInDomain)$,
then we say that \emph{the singular values of the sequence $(A_n)_{n=1}^\nf$
are asymptotically distributed as $(\LimFunction,\ProbInDomain)$}.
A similar terminology is used for the eigenvalues.

\subsection*{\bf Multilevel Toeplitz matrices}

Let $a$ be a function in $L^\infty$ on $\bT^k$, where $\bT$ is the complex unit circle. The Fourier coefficients of
$a$ are defined by
\[a_{j_1,\ldots,j_k}=\frac{1}{(2\pi)^k}\int_0^{2\pi}\cdots\int_0^{2\pi}a(e^{i\tht_1},\ldots,e^{i\tht_k})e^{-i(j_1\tht_1+\cdots+j_k\tht_k)}
d\tht_1\cdots d\tht_k.\]
Suppose that for each $n \in \bN$ we are given a $k$-tuple $(m_1^{(n)}, \ldots,m_k^{(n)}) \in \bN^k$.
We denote by $W_n(a)$ the linear operator acting on
\[Y_n:=\ell^2(\{1,\ldots,m_1^{(n)}\} \times \cdots \times \{1,\ldots,m_k^{(n)}\})\]
by the rule
\[(W_n(a)x)_{i_1,\ldots,i_k}=\sum_{j_1=1}^{m_1^{(n)}}\cdots \sum_{j_k=1}^{m_k^{(n)}}
a_{i_1-j_1, \ldots, i_k-j_k}x_{j_1, \ldots, j_k},\]
and we let $T_n(a)$ stand for the matrix representation of $W_n(a)$ in the standard basis of $Y_n$.
The matrix $T_n(a)$ is called a $k$-level Toeplitz matrix. Note that in this case $d(n)=m_1^{(n)}\cdots m_k^{(n)}$.
Tyrtyshnikov~\cite{Tyrt1996} showed that
if \[\min(m_1^{(n)}, \ldots,m_k^{(n)}) \to \infty \;\:\mbox{as}\;\: n\to \infty,\] then the singular values of $T_n(a)$ are
asymptotically distributed as $X:=|a|$ on $\bT^k$ with normalized invariant measure. In~\cite{BBGM2015maximum},
we showed that if the essential range $\Range(|a|)$
is just the segment $[0, \|a\|_\infty]$,
then
\[\lim_{n \to \infty}\max_{1 \le j \le d(n)} |s_j(T_n(a))-Q_{|a|}(j/d(n))|=0.\]
Now this result can simply be deduced from Tyrtyshnikov's in conjunction with Theorem 1.2.

\subsection*{\bf Sums of products of Toeplitz matrices}

A $1$-level Toeplitz matrix is a usual Toeplitz matrix, that is, a matrix of the form $(a_{i-j})_{i,j=1}^n$.
If the entries $a_k$ ($k \in \bZ$) are  the Fourier coefficients of a function $a \in L^\infty(\bT)$,
then $(a_{i-j})_{i,j=1}^n$ is denoted by $T_n(a)$ and $a$ is referred to as the symbol of
the matrices $T_n(a)$ ($n \in \bN$).

\vsk
Denote by $\mu_\bT$ the normalized invariant measure on the unit circle $\bT$.
For every pair $(p,q)$ with $p\in\{1,\ldots,M\}$, $q\in\{1,\ldots,N_p\}$,
take functions $a^{(p,q)}\in L^\nf(\bT)$
and define $B_n\in\bC^{n\times n}$ by
\[
B_n = \sum_{p=1}^M \prod_{q=1}^{N_p} T_n(a^{(p,q)}).
\]
Then it is known from \cite{BS1999,Serra2003,Strouse,Tyrt1994,Tyrt2001}
that the singular values of $B_n$ are asymptotically distributed as $(\LimFunction,\mu_\bT)$, where
\[
\LimFunction := \left|\sum_{p=1}^M \prod_{q=1}^{N_p} a^{(p,q)}\right|.
\]
If $\Range(\LimFunction)$ is a segment $[0,\be]$
and $\|B_n\|\le\be$ for every $n\in\bN$,
then Theorem~\ref{thm:tuples} assures that
\begin{equation}\label{eq:sumproductapprox}
\lim_{n\to\nf} \max_{1\le j\le n} |s^{(n)}_{j}(B_n)-\LimFunction(j/n)| = 0.
\end{equation}

As an example, consider the products $B_n=T_n(a^{(1)})T_n(a^{(2)})$
of Toeplitz matrices with the symbols $a^{(1)}$ and $a^{(2)}$ shown in Figure~\ref{fig:product}.
In that case $\LimFunction=a^{(1)} a^{(2)}$,
the essential range of $\LimFunction$ is $[0,1]$,
and the norms of $B_n$ are bounded by $1$.
Therefore \eqref{eq:sumproductapprox} holds.
Note that $\Range(a^{(1)})$ and $\Range(a^{(2)})$ have gaps,
which implies that the singular values of $T(a^{(q)})$
cannot be approximated uniformly  by the values of $a^{(q)}$, $q=1,2$.
Denoting the maximum on the left-hand side of \eqref{eq:sumproductapprox} by $\eps^{(n)}$
we get the following table:
\[
\begin{array}{c|c|c|c|c|c|c}
\displaystyle\vphantom{0_{0_0}^{0^0}} n & 32 & 64 & 128 & 256 & 512 & 1024 \\\hline
\displaystyle\vphantom{0_{0_0}^{0^{0^0}}} \eps^{(n)} &
\ 5.7\cdot 10^{-2}\ &\ 3.1\cdot 10^{-2}\ &\ 1.6\cdot 10^{-2}\ &\ 8.6\cdot 10^{-6}\ &\ 4.4\cdot 10^{-3}\ &\ 2.3\cdot 10^{-3}\
\end{array}.
\]

\begin{figure}[ht]
\centering
\includegraphics{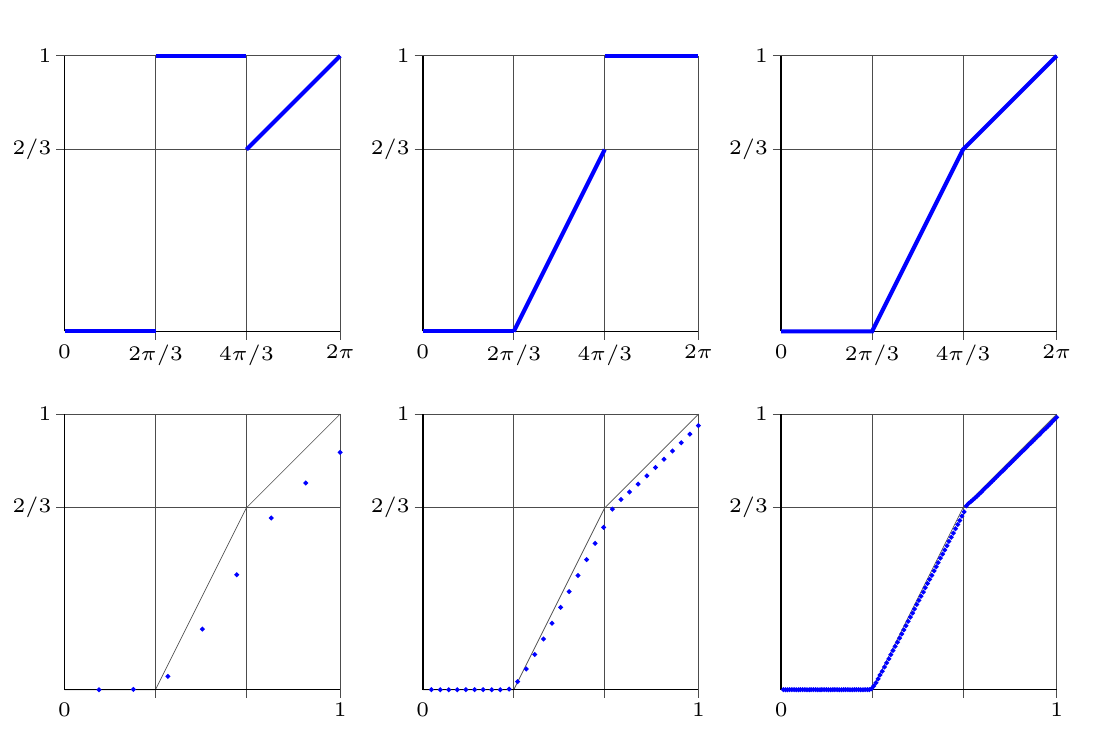}
\caption{The top row shows the symbols $a^{(1)}$, $a^{(2)}$, and $a^{(1)}a^{(2)}$, respectively.
The bottom row shows the singular values of $B_n$ for $n=8,32,128$
as points with coordinates $(j/n, s^{(n)}_j)$, and the quantile function of $a^{(1)}a^{(2)}$.}
\label{fig:product}
\end{figure}

\subsection*{\bf Other Toeplitz-like matrices}

The asymptotic distribution of singular and eigenvalues are known
for many other classes of matrices:
$g$-Toeplitz matrices \cite{ENSS2012,NSS2010},
locally Toeplitz matrices \cite{Serra2003,Tilli1998,ZS2004},
and also in the more general situation when some ``complicated'' matrices
can be approximated by ``simple'' matrices with known distribution of singular values or eigenvalues
(see \cite{GSV2015} or \cite[Theorem 2.1]{Tyrt1994}).

\vsk
In all these cases, the uniform convergence of the eigenvalues holds
under the assumption that the matrices $A_n$ are Hermitian, that the essential range of the function $\LimFunction$ is bounded and connected,
and that the eigenvalues of $A_n$ are contained in $\Range(\LimFunction)$.
For the singular values of (not necessarily Hermitian) matrices $A_n$, it is sufficient to require
that $\Range(\LimFunction)$ is a segment of the form $[0,\be]$
and that $\|A_n\|\le\be$ for every $n$.

\vsk
We remark that $Q_{\LimFunction}(j/d(n))$ is a very rough approximation
to individual singular values (or eigenvalues) because
the magnitude of the error is usually comparable with the distance between consecutive singular values
(or eigenvalues).
However, the quantile approach yields more precise approximations
once the sites $j/d(n)$ are substituted by more cleverly chosen points; see \cite{BBGM2015simpleloop}.
We will not embark on this subtle issue here.
Theorem~\ref{thm:Riemann} should nevertheless be of use for numerical methods
since it has the potential to provide us with an initial approximation for iterative algorithms.

\subsection*{\bf Nets instead of sequences}

Theorems~\ref{thm:measures}, \ref{thm:tuples}, \ref{thm:Riemann}
are also true with sequences replaced by nets.
We decided to restrict ourselves to sequences just for the sake of simplicity.
But here is a result by I.~B.~Simonenko \cite{Sim}
where nets are the appropriate language.

\vsk
Let $a: \bT^k \to \bR$ be a continuous function.
For a finite subset $M$ of $\bZ^k$, denote by $W_M(a)$ the linear operator defined
by
\[(W_M(a)x)_i=\sum_{j \in M} a_{i-j}x_j, \quad i \in M\]
on $\ell^2(M)$ and let $T_M(a)$ be the matrix representation of $W_M(a)$ in the standard basis
of $\ell^2(M)$. Now suppose $(M_\nu)_{\nu \in \cN}$ is any net
of finite subsets of $\bZ^k$ such that
\[
\lim_{\nu\in\cN} \min_{Y\subset M_\nu}
\max\left(\frac{\# (M_\nu \setminus Y)}{\# M_\nu}, \frac{1}{\dist(Y, \bZ\setminus M_\nu)}\right) = 0.
\]
Simonenko showed that then the eigenvalues $\la_j(T_{M_\nu}(a))$ of the Hermitian matrices $T_{M_\nu}(a)$ all belong to
$\Range(a)=[\min(a),\max(a)]$ and are asymptotically distributed as $X:=a$ on $\bT^k$ with
the measure $\mu_\bT\times \cdots \times \mu_\bT$. From Theorem~1.2 (for nets) we therefore conclude that
\[\lim_{\nu \in \cN} \max_{1 \le j \le \#M_\nu}|\la_j(T_{M_\nu}(a))-Q_a(j/\#M_\nu)|=0.\]

\section{Examples from beyond the matrix world}
\label{sec:applsequences}

\begin{example}\label{example:arcsin}
{\rm The purpose of this example is to turn inside out the
famous \emph{arcsine law} for random walks
discovered by P.~L\'{e}vy in 1939 (see \cite[Chapter~10]{Lesigne2005}).
What results after that procedure is the quantile version of the arcsine law, which might be called
the  \emph{sine law} for random walks.

\vsk
For every $n\in\bN$ and every $w=(w_1,\ldots,w_n)\in\{-1,1\}^n$, put
\[
G_n(w)=\frac{\#\{k\in\{1,\ldots,n\}\colon\ w_1+\ldots+w_k > 0\}}{n}.
\]
Let $\tuplesize(n)=2^n$ and $\al^{(n)}_1,\ldots,\al^{(n)}_{\tuplesize(n)}$ be the numbers
$G_n(w)$, $w\in\{-1,1\}^n$, written in the increasing order.
For example, if $n=3$, then we have $2^3=8$ elements $w$ with the following values of $G_3(w)$:
\begin{eqnarray*}
& &
G_3(\{-1,-1,-1\})=0, \quad
G_3(\{-1,-1,1\})=0, \quad
G_3(\{-1,1,-1\})=0, \\
& & G_3(\{-1,1,1\})=\frac{1}{3}, \quad
G_3(\{1,-1,-1\})=\frac{1}{3}, \quad
G_3(\{1,-1,1\})=\frac{2}{3}, \\
& & G_3(\{1,1,-1\})=1, \quad
G_3(\{1,1,1\})=1.
\end{eqnarray*}
The corresponding collection $\al^{(3)}=(\al^{(3)}_1,\ldots,\al^{(3)}_{8})$ is
\[
\left(0,0,0,\frac{1}{3},\frac{1}{3},\frac{2}{3},1,1\right).
\]
L\'{e}vy's arcsine law says that $\al^{(n)}=(\al^{(n)}_1,\ldots,\al^{(n)}_{\tuplesize(n)})$ is asymptotically distributed as $(\LimFunction,\ProbInDomain)$,
where $\LimFunction$ is defined on $[0,1]$ by
\[
\LimFunction(v) \eqdef \frac{2}{\pi}\arcsin\sqrt{v},
\]
and $\ProbInDomain$ is the Lebesgue measure on $[0,1]$.
Consequently,
\[
Q_{\LimFunction}(p)=\sin^2\frac{\pi p}{2},
\]
and by Theorem~\ref{thm:tuples},
\begin{equation}\label{eq:limsinlaw}
\lim_{n\to\nf}\max_{1\le j\le 2^n} \left|\al^{(n)}_j-\sin^2\frac{\pi j}{2^{n+1}}\right| = 0.
\end{equation}
Figure~\ref{fig:arcsin} shows the values $\al^{(n)}_j$ for $n=3$ and $n=30$.
We see that the convergence is very slow.
Denoting  the maximum on the left-hand side of \eqref{eq:limsinlaw} by $\eps^{(n)}$,
we get the following table:
\[
\begin{array}{c|c|c|c|c|c|c}
\displaystyle\vphantom{0_{0_0}^{0^0}} n & 5 & 10 & 15 & 20 & 25 & 30 \\\hline
\displaystyle\vphantom{0_{0_0}^{0^{0^0}}} \eps^{(n)} &\ 0.300\ &\ 0.209\ &\ 0.164\ &\ 0.144\ &\ 0.126\ &\  0.116\
\end{array}.
\]

\begin{figure}[ht]
\centering
\includegraphics{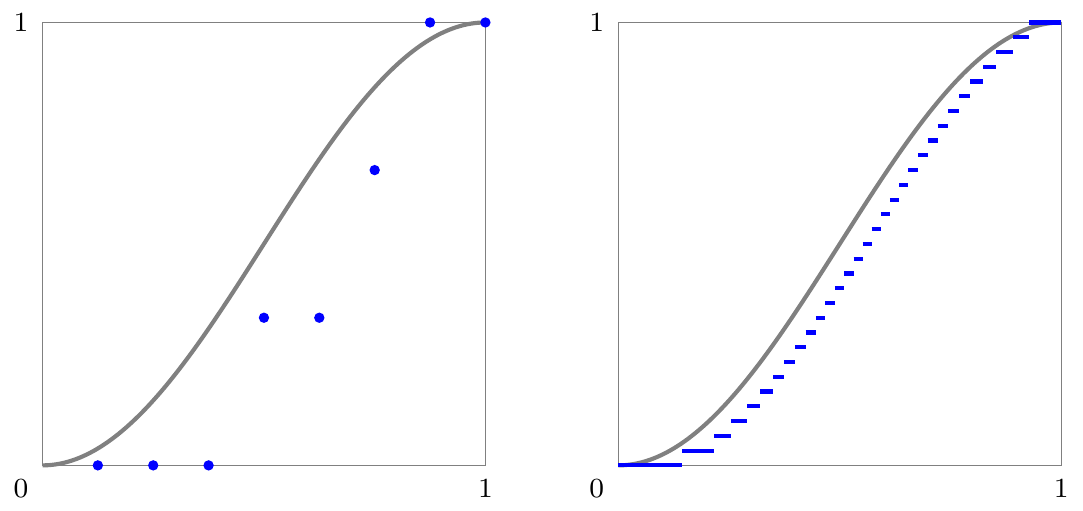}
\caption{The left picture shows the points $(j/2^n,\al^{(n)}_j)$ for $n=3$
and the plot of $Q_{\LimFunction}(v)$ from Example~\ref{example:arcsin}.
The right picture corresponds to $n=30$; we there glued together $2^{30}$ points to $31$ line segments.
}
\label{fig:arcsin}
\end{figure}

}
\end{example}

\vsg
The next result follows from Theorem~\ref{thm:measures}
and indicates another class of applications of that theorem,
namely, application to \emph{asymptotically distributed sequences of numbers}.

\begin{proposition}\label{prop:asympt_distr_seq}
Let $\LimProb$ be a Borel probability measure on $\bR$ with bounded connected support
and let $(\be_j)_{j=1}^\nf$ be a bounded sequence of real numbers which are
asymptotically distributed as $\LimProb$ in the sense that, for every $v\in\bR$,
\[
\lim_{n\to\nf}\frac{\#\{j\in\{1,\ldots,n\}\colon\ \be_j\le v\}}{n}=F_\LimProb(v).
\]
Let $(\al^{(n)}_1,\ldots,\al^{(n)}_n)$ denote
the collection $(\be_1,\ldots,\be_n)$ written in the ascending order.
Then
\[
\lim_{n\to\nf} \max_{1\le j\le n} |\al_j^{(n)} - Q_\LimProb(j/n)| = 0.
\]
\end{proposition}

In particular, for sequences which are \emph{uniformly distributed} on $[0,1]$ (see~\cite{KN1974}),
one has $Q(p)=p$ for every $p\in[0,1]$,
and hence Proposition~\ref{prop:asympt_distr_seq} implies that
\begin{equation}\label{eq:uniconverg_unidistr}
\lim_{n\to\infty} \max_{1\le j\le n} |\al_j^{(n)} - j/n| = 0.
\end{equation}

\begin{example}\label{Weyl}
{\rm
Consider the sequence $\be_j=j\sqrt{2}-\lfloor j\sqrt{2}\rfloor$.
By Weyl's equidistribution theorem, it is uniformly distributed on $[0,1]$.
Figure \ref{fig:Weyl} shows the points $(j/64,\al_j^{(64)})$ for $j=1,\ldots,64$.
Denoting the maximum on the left-hand side of \eqref{eq:uniconverg_unidistr} by $\eps^{(n)}$
we get the following table:
\[
\begin{array}{c|c|c|c|c|c|c}
\displaystyle\vphantom{0_{0_0}^{0^0}} n & 32 & 64 & 128 & 256 & 512 & 1024 \\\hline
\displaystyle\vphantom{0_{0_0}^{0^{0^0}}} \eps^{(n)} &
\ 4.6\cdot 10^{-2}\ &\ 2.6\cdot 10^{-2}\ &\ 1.1\cdot 10^{-2}\ &\ 5.6\cdot 10^{-2}\ &\ 3.3\cdot 10^{-3}\ &\  2.4\cdot 10^{-3}\
\end{array}.
\]
If fact, the behavior of $\eps^{(n)}$ is rather irregular,
but we obtained $\eps^{(n)} \le \dfrac{0.7\ln(n)}{n}$ for $n=2,3,\ldots,10000$.

\begin{figure}[ht]
\centering
\includegraphics{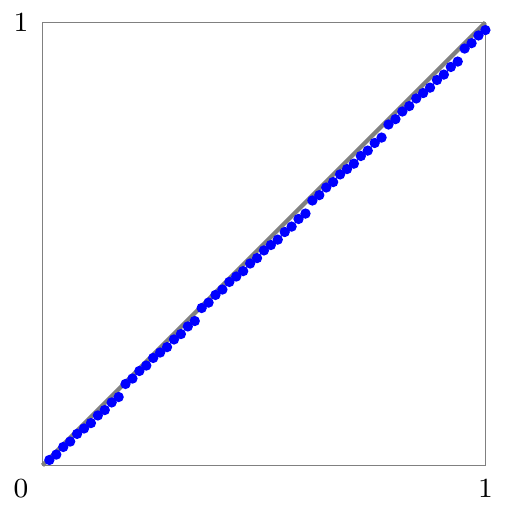}
\caption{The blue points are the points whose abscissas are $j/64$ and whose ordinates are the ordered numbers $j\sqrt{2}-\lfloor j\sqrt{2}\rfloor$ with $j=1,\ldots,64$.
The gray line is the graph of the identity function on $[0,1]$.}
\label{fig:Weyl}
\end{figure}

}
\end{example}


\begin{thebibliography}{20}

\bibitem{Bogachev2007}
Bogachev, V.I.:
{\em Measure Theory, Volume II.}
Springer-Verlag, Berlin and Heidelberg (2007).

\bibitem{BBGM2015maximum}
Bogoya, J.M., B\"{o}ttcher, A., Grudsky, S.M., Maximenko, E.A.:
{\em Maximum norm versions of the Szeg\H{o} and Avram-Parter theorems for Toeplitz matrices.}
J. Approx. Theory  196, 79--100 (2015).
\doi{10.1016/j.jat.2015.03.003}

\bibitem{BBGM2015simpleloop}
Bogoya, J.M., B\"{o}ttcher, A., Grudsky, S.M., Maximenko, E.A.:
{\em Eigenvalues of Hermitian Toeplitz matrices with smooth simple-loop symbols.}
J. Math. Anal. Appl. 422, 1308--1334 (2015).
\doi{10.1016/j.jmaa.2014.09.057}

\bibitem{BG2005}
B\"{o}ttcher, A., Grudsky, S.M.:
{\em Spectral Properties of Banded Toeplitz Matrices.}
SIAM, Philadelphia (2005).

\bibitem{BGM2010}
B\"{o}ttcher, A., Grudsky, S., Maksimenko, E.A.:
{\em Inside the eigenvalues of certain Hermitian Toeplitz band matrices.}
J. Comput. Appl. Math. 233, 2245--2264 (2010).
\doi{10.1016/j.cam.2009.10.010}

\bibitem{BS1999}
B\"{o}ttcher, A., Silbermann, B.:
{\em Introduction to Large Truncated Toeplitz Matrices.}
Springer-Verlag, New York (1999).

\bibitem{DIK2012}
Deift, P., Its, A., Krasovsky, I.:
{\em Eigenvalues of Toeplitz matrices in the bulk of the spectrum.}
Bull. Inst. Math. Acad. Sin. (N.S.) 7, 437--461 (2012).
% URL: \texttt{http://w3.math.sinica.edu.tw/bulletin\_ns/20124/2012401.pdf}

\bibitem{ENSS2012}
Estatico, C., Ngondiep, E., Serra-Capizzano, S., Sesana, D.:
{\em A note on the {\rm (}regularizing{\rm )} preconditioning of $g$-Toeplitz sequences via $g$-circulants.}
J. Comp. Appl. Math. 236, 2090--2111 (2012).
\doi{10.1016/j.cam.2011.09.033}

\bibitem{GSV2015}
Garoni, C., Serra-Capizzano, S., Vassalos, P.:
{\em A general tool for determining the asymptotic spectral
distribution of Hermitian matrix-sequences.}
Oper. Matrices 9, 549--561 (2015).
\doi{10.7153/oam-09-33}
%\href{http://files.ele-math.com/preprints/oam-09-33.pdf}{http://files.ele-math.com/preprints/oam-09-33.pdf}

\bibitem{KN1974}
Kuipers, L., Niederreiter, H.:
{\em Uniform Distribution of Sequences.}
Wiley, New York, London, Sydney (1974).

\bibitem{Lesigne2005}
Lesigne, E.:
{\em Heads or Tails: An Introduction to Limit Theorems in Probability.}
Student Mathematical Library, Vol. 28, AMS, Providence (2005).

\bibitem{NSS2010}
Ngondiep, E., Serra-Capizzano, S., Sesana, D.:
{\em Spectral features and asymptotic properties for $g$-circulants and $g$-Toeplitz sequences.}
SIAM J. Matrix Anal. Appl. 31, 1663--1687 (2010).
\doi{10.1137/090760209}

\bibitem{Szeg1}
Szeg\H{o}, G.:
{\em Ein Grenzwertsatz \"uber die Toeplitzschen Determinanten einer reellen positiven Funktion.}
Math. Ann. 76, 490--503 (1915).
\doi{10.1007/BF01458220}

\bibitem{Szeg2}
Szeg\H{o}, G.:
{\em Beitr\"age zur Theorie der Toeplitzschen Formen I.}
Math. Zeitschr. 6, 167--202 (1920).
\doi{10.1007/BF01199955}

\bibitem{Serra2003}
Serra~Capizzano, S.:
{\em Generalized locally Toeplitz sequences: spectral analysis and applications to discretized partial differential equations.}
Linear Alg. Appl. 366, 371--402 (2003).
\doi{10.1016/S0024-3795(02)00504-9}

\bibitem{Sim}
Simonenko, I.B.:
{\em Szeg\H{o} type limit theorems for multidimensional discrete convolution operators with continuous symbols.}
Funct. Anal. Appl. 35, 77--78 (2001).
\doi{10.1023/A:1004136903704}

\bibitem{Strouse}
Serra-Capizzano, S., Sesana, D., Strouse, E.:
{\em The eigenvalue distribution of products of Toeplitz matrices -- clustering and attraction.}
Linear Alg. Appl. 432, 2658--2687 (2010).
\doi{10.1016/j.laa.2009.12.005}

\bibitem{Tilli1998}
Tilli, P.:
{\em Locally Toeplitz sequences: spectral properties and applications.}
Linear Alg. Appl. 278, 91--120 (1998).
\doi{10.1016/S0024-3795(97)10079-9}

\bibitem{Trench2012}
Trench, W.F.:
{\em An elementary view of Weyl's theory of equal distribution.}
Amer. Math. Monthly 119, 852--861 (2012).
\doi{10.4169/amer.math.monthly.119.10.852}

\bibitem{Tyrt1994}
Tyrtyshnikov, E.E.:
{\em Influence of matrix operations on the distribution of eigenvalues and singular values of Toeplitz matrices.}
Linear Alg. Appl. 207, 225--249 (1994).
\doi{10.1016/0024-3795(94)90012-4}

\bibitem{Tyrt1996}
Tyrtyshnikov, E.E.:
{\em A unifying approach to some old and new theorems on distribution and clustering.}
Linear Alg. Appl. 232, 1--43 (1996).
\doi{10.1016/0024-3795(94)00025-5}

\bibitem{Tyrt2001}
Tyrtyshnikov, E.E.:
{\em Some applications of a matrix criterion for equidistribution.}
Mat. Sb. 192 (12), 1877--1887 (2001).
\doi{10.1070/SM2001v192n12ABEH000618}

\bibitem{Vaart}
Vaart, A.W. van der:
{\em Asymptotic Statistics.}
Cambridge University Press, Cambridge (1998).

\bibitem{ZS2004}
Zabroda, O.N., Simonenko, I.B.:
{\em Asymptotic invertibility and the collective asymptotic
spectral behavior of generalized one-dimensional discrete convolutions.}
Funct. Anal. Appl. 38, 65--66 (2004).
\doi{10.1023/B:FAIA.0000024869.01751.f0}

\end{thebibliography}
\end{document}